\documentclass[11pt,letterpaper]{amsart}
\usepackage{extarrows}
\usepackage{tikz-cd}
\usetikzlibrary{matrix}

\newtheorem{theoremA}{Theorem}

\renewcommand{\thetheoremName}

\newtheorem{theorem}{Theorem}[section]
\newtheorem*{theorem*}{Theorem}
\newtheorem{lemma}[theorem]{Lemma}
 \newtheorem{corollary}[theorem]{Corollary}
 
 \newtheorem*{mainthm*}{Main Theorem}

\theoremstyle{definition}
\newtheorem{definition}[theorem]{Definition}

\theoremstyle{remark}
\newtheorem{remark}[theorem]{Remark}

\definecolor{alert}{rgb}{0.8,0,0.3}
\newcommand{\alert}[1]{%
	\marginpar{%
		\ifodd\value{page} \raggedright \else \raggedleft \fi
		\footnotesize{\textcolor{alert}{#1}}
	}
}

\newcommand{\Div}{\operatorname{div}}

\newcommand{\erre}{\mathbb{R}}
\newcommand{\ese}{\mathbb{S}}

\begin{document}

\title[Fat equator effect and minimality]{Fat equator effect and Minimality in immersions and submersions of the Sphere}

%    Remove any unused author tags.

%    author one information

%    author two information
\author{Vicent Gimeno i Garcia}
\address{Department of Mathematics, Universitat Jaume I-IMAC,   E-12071, 
Castell\'{o}, Spain}
%\curraddr{}
\email{gimenov@mat.uji.es}
\author{Vicente Palmer}
\address{Department of Mathematics, Universitat Jaume I-INIT,   E-12071, 
Castell\'{o}, Spain}
%\curraddr{}
\email{palmer@mat.uji.es}

\thanks{Research partially supported by the Research Program of University Jaume I Project UJI-B2021-08, and and Research grant  PID2020-\-115930\-GA-100 funded by MCIN/ AEI /10.13039/501100011033}



%\date{\today}

\dedicatory{}

%%%%%%%%%%%%%%%%%%
%Abstract
%%%%%%%%%%%%%%%%%%%%
\begin{abstract}
Inspired by the equatorial concentration of measure phenomenon in the sphere, a result which is deduced from the general, (and intrinsic), concentration of measure in $\ese^n(1)$, we describe in this paper an equatorial concentration of measure satisfied by the closed, (compact without boundary), isometric and minimal immersions $x:\Sigma^m \rightarrow \ese^n(1)$, ($m \leq n$), and by the minimal Riemannian submersions $\pi: \Sigma^m \rightarrow \ese^n(1)$, ($m \geq n$).\end{abstract}

\maketitle

%%%%%%%%%%%%%%%%%%
%Section: Introduction
%%%%%%%%%%%%%%%%%%%%
\section{Introduction}\label{sec:intro}\
%%%%%%%%%%%%%%%%%%%%%%

The deep connection between the notions of measure of sets, (in its meaning of Riemannian volume), distance between points and dimension in a Riemannian manifold is illustrated by the concentration of (intrinsic) measure phenomenon in the sphere $\ese^n(1)$, (that we consider in this paper equipped with a metric of constant sectional curvature $1$), as we can find it in the lecture notes \cite{MS}, (see Theorem \ref{measurecon} in this Introduction).

From this concentration of measure  it can be deduced a particular result, namely, an intrinsic {\em equatorial} measure concentration property satisfied by the strips around the equators of the sphere, also known as \lq\lq fat equator effect'', which can be roughly described by saying that for ``very large" dimension, almost all measure in the sphere concentrates in these strips surrounding its equators, independently of its width, (see Theorems \ref{thB} and Theorem \ref{teo:mitjana} in this Section). 

Now, an interesting question that can be raised from this last result appears when considering it from an extrinsic point of view, that is, from the point of view of the submanifold theory. 
We introduce at this point  the concept of \lq\lq {\em extrinsic} equatorial measure concentration"  for submanifolds in the sphere: we shall study then the relative volume of the intersection of a closed and minimal submanifold with the above mentioned strips around the equators of the ambient sphere, concluding, as in the intrinsic case, that for \lq\lq very large" dimension, almost all measure in the submanifold concentrates in their intersection with the strips of any width surrounding these equators, (see Theorem \ref{mainInt} in this Introduction).

On the other hand, there is a concentration of measure in the dual geometric setting with respect the isometric immersions given by the Riemmanian submersions. Namely, we shall prove a concentration of measure satisfied by the compact manifolds which admits a Riemannian submersion onto the sphere with minimal fibers. In this vein, such compact manifolds
 accumulates almost the whole measure on the  points which project to the strips of any width surrounding the equator of the sphere (see Theorem \ref{TeoSub}).

In the lecture notes by V. Milman and G. Schechtman \cite{MS} we can find the following theorem:

\begin{theoremA}[Corollary 2.2 in \cite{MS}]\label{measurecon}\

 If $A\subset \mathbb{S}^n(1)$ is a domain with ${\rm vol }(A)\geq \frac{1}{2} {\rm vol}(\mathbb{S}^n(1))$ then, for all $\epsilon \in [0,\frac{\pi}{2}]$, the  $\epsilon$-fattening of $A$ satisfies the inequality
\begin{equation}\label{eq1}
1 \geq \frac{{\rm vol}(A_\epsilon)}{{\rm vol}(\mathbb{S}^n(1))}\geq 1-\sqrt{\frac{\pi}{8}}e^{-\epsilon^2(n-1)/2}
\end{equation}
\noindent and hence, for all $\epsilon \in [0,\frac{\pi}{2}]$, 
\begin{equation}\label{eq2}
\lim_{n\to \infty} \frac{{\rm vol}(A_\epsilon)}{{\rm vol}(\mathbb{S}^n(1))}= 1.
\end{equation}
\end{theoremA}  
\begin{remark}
Here, the $\epsilon$-fattening of $A$ is defined, using the intrinsic distance in the sphere $\ese^n(1)$, as
$$A_\epsilon:=\left\{x\in \mathbb{S}^n(1)\, :\, {\rm dist}^{\mathbb{S}^n(1)}(x,A)<\epsilon\, \right\}
$$
 \end{remark}
  
 Theorem \ref{measurecon} is a consequence of the Levy's isoperimetric inequality in the sphere, (see \cite{MS}).
  To see the idea behind the proof of assertions (\ref{eq1}) and (\ref{eq2}), let us consider a  domain $A$ with ${\rm vol }(A)=a \geq \frac{1}{2} {\rm vol}(\mathbb{S}^n(1))$ and let us fix $\epsilon \in [0,\frac{\pi}{2}]$. Then,  applying Levy's isoperimetric inequality, and having into account that, fixed $\epsilon \in [0,\frac{\pi}{2}]$, the $\epsilon$-fattening of a geodesic ball $B^{n,1}_{r}$ is the ball $(B^{n,1}_{r})_{\epsilon}=B^{n,1}_{r+\epsilon}$, for all $r \in [0,\pi/2-\epsilon)$, we have the inequality
  
\begin{equation}
\begin{aligned}
\frac{{\rm vol}\left(A_\epsilon\right)}{{\rm vol}\left(\mathbb{S}^n(1)\right)}&\geq  \frac{{\rm vol}\left(B^{n,1}_{\pi/2+\epsilon}\right)}{{\rm vol}\left(\mathbb{S}^n(1)\right)}= \frac{\int _0^{\frac{\pi}{2}+\epsilon} \sin^n t dt }{\int _0^\pi \sin^n t dt }\\
&=1-\frac{\int_{\frac{\pi}{2}+\epsilon}^{\pi} \sin^n t dt }{\int _0^\pi \sin^n t dt }
\end{aligned}
\end{equation}

To obtain equality (\ref{eq2}) from this point, we remark that , given $\epsilon \in [0,\frac{\pi}{2}]$,  in the intervals $[0, \frac{\pi}{2}+\epsilon] \subseteq [0,\pi]$, the function $\sin t$ attains its maximum of $1$ in $t=\frac{\pi}{2}$, so, when $n \to \infty$, the function $\sin^n t$ has a ``peak'' at $\frac{\pi}{2}$ and is negligible out of this value. Hence, we have that
\begin{equation}
%\begin{aligned}
 \lim_{n \to \infty}\frac{\int_{\frac{\pi}{2}+\epsilon}^{\pi} \sin^n t dt }{\int _0^\pi \sin^n t dt }=0
%1\geq \lim_{n \to \infty} \frac{{\rm vol}\left(A_\epsilon\right)}{{\rm vol}\left(\mathbb{S}^n(1)\right)}&\geq  \lim_{n \to \infty} \frac{\int _0^{\frac{\pi}{2}+\epsilon} (\sin t)^n dt }{\int _0^\pi (\sin t)^n dt   }=1
%\end{aligned}
\end{equation}
\noindent and
\begin{equation}
%\begin{aligned}
% \lim_{n \to \infty}\frac{\int_{\frac{\pi}{2}+\epsilon}^{\pi} \sin^n t dt }{\int _0^\pi \sin^n t dt }&=0\\
1\geq \lim_{n \to \infty} \frac{{\rm vol}\left(A_\epsilon\right)}{{\rm vol}\left(\mathbb{S}^n(1)\right)}\geq  \lim_{n \to \infty} \frac{\int _0^{\frac{\pi}{2}+\epsilon} (\sin t)^n dt }{\int _0^\pi (\sin t)^n dt   }=1
%\end{aligned}
\end{equation}
\medskip

Now, we are going to focus on a particular concentration of measure phenomenon in the sphere, that we call {\em equatorial}, what happens in this case around any totally geodesic equator $\mathbb{S}^{n-1}(1)$ of the sphere $ \mathbb{S}^{n}(1)$, and which can be deduced from the general concentration of measure described above. We will take  it as  a starting point for our definition of the extrinsic concentration of measure of an immersed submanifold in the sphere.

In order to describe it and to give an idea of its proof, we are going to consider the particular case constituted by the equator 
$$E=\{(x_1,\ldots,x_{n+1})\in \mathbb{S}^{n}(1)\, :\, x_{n+1}=0\}$$
together the domain $\Omega_\epsilon \subseteq  \mathbb{S}^{n}(1) $, (called $\epsilon$-strip around the equator $E$, see also figure \ref{fig:e-strip}), given by

\begin{equation*}
\Omega_\epsilon:=\left\lbrace  (x_1,\ldots,x_{n+1})\in \mathbb{S}^{n}(1)\, :\, -\sin(\epsilon)<x_{n+1}<\sin(\epsilon)\right\rbrace
\end{equation*}

\noindent where $\epsilon \in [0,\frac{\pi}{2}]$

We must remark at this point that any equator in the sphere can be placed in the position of $E$ using the accurated rotation, (an isometry which preserves volume and distances), so the argument we are going to show applies in fact for any equator.

\begin{center}
\begin{figure}
    \centering
    \includegraphics[scale=0.35]{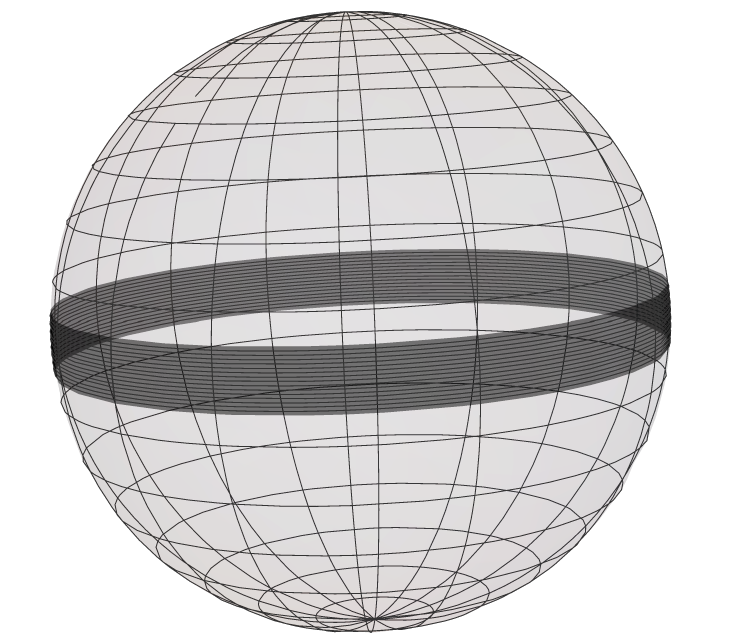}
    \caption{The $\epsilon$-strip of $\mathbb{S}^2(1)$ is the $\epsilon$-fattening of the equator of $\mathbb{S}^2(1)$}\label{fig:e-strip}
\end{figure}
\end{center}

Then, we can prove the following
\begin{theoremA}\label{equator}
Given the $\epsilon$-strip $\Omega_\epsilon$ around the equator $E$, 
\begin{equation}\label{eq2.5}
\frac{{\rm vol}\left(\Omega_\epsilon\right)}{{\rm vol}\left(\mathbb{S}^n(1)\right)}\geq 1-\sqrt{\frac{\pi}{2}}e^{-\epsilon^2\frac{n-1}{2}}
\end{equation}
 and hence, for $0<\epsilon<\frac{\pi}{2}$ we have
\begin{equation}\label{eq2.5.1}
\lim_{n\to \infty}\frac{{\rm vol}\left(\Omega_\epsilon\right)}{{\rm vol}\left(\mathbb{S}^n(1)\right)}=1.
\end{equation}
\end{theoremA}

This property comes from the concentration of measure stated in Theorem \ref{measurecon} in the following way: let us denote as $\mathbb{S}^{n+}$ the half sphere centered at the north pole and  as $\mathbb{S}^{n-}$ the half sphere centered at the south pole. Then, the $\epsilon$-strip around the equator can be expressed, in terms of the $\epsilon$-fattenings $\mathbb{S}^{n+}_{\epsilon}$ and $\mathbb{S}^{n-}_{\epsilon}$ as the intersection, 
$$\Omega_\epsilon=\mathbb{S}^{n+}_{\epsilon} \cap\mathbb{S}^{n-}_{\epsilon} $$

%\begin{figure}
 %   \centering
  %  \includegraphics[scale=0.25]{positiva}\quad\includegraphics[scale=0.3]{negativa}
   % \caption{The $\epsilon$-strip is the intersection between the $\epsilon$-fattening of the north hemisphere $\mathbb{S}^{+}_\epsilon$ and the $\epsilon$-fattening of the north hemisphere $\mathbb{S}^{-}_\epsilon$.}
   % \label{fig:asintersection}
%\end{figure}

\noindent Now, inequality (\ref{eq2.5}) follows applying Theorem \ref{measurecon} to the $\epsilon$-fattenings $\mathbb{S}^{n+}_{\epsilon}$ and $\mathbb{S}^{n-}_{\epsilon}$ and using the fact that ${\rm vol}(\mathbb{S}^{n+}_\epsilon)={\rm vol}(\mathbb{S}^{n-}_\epsilon)$.The conclusion is that for  \lq\lq very large'' dimension,  almost all measure in the Sphere is concentrated surrounding  the equator. 

This equatorial concentration of measure just described above is also known as \lq\lq fat equator effect'' (see \cite{Bengtsson}) 
and can be alternatively deduced by using another concentration of measure, namely, the concentration of measure around the mean/averaged value of a Lipschitz function defined on the sphere $\ese^n(1)$, as it is stated in Corollary V.2 of \cite{MS}:

\begin{theoremA}[Corollary V.2 in \cite{MS}]\label{thB}
Let $f: \ese^n(1) \rightarrow \erre$ be a Lipschitz function, with Lipschitz constant $\sigma$. Then there exists an absolute constant $\delta$ such that
\begin{equation}
\frac{{\rm vol}(\{ x \in \ese^n(1) / \vert f(x)- \frac{\int_{\ese^n} f d\mu}{{\rm vol}(\ese^n)}\vert\geq C\} }{{\rm vol}(\ese^n) } \geq 4 e^{-\delta C^2\frac{n+1}{\sigma^2}}
\end{equation}

\end{theoremA}
We must remark that Theorem \ref{thB} is deduced from a previous result due to B. Maurey and G. Pisier, (see Theorem V.1 in \cite{MS}). In this theorem it is proved a concentration of measure of a multivariate and Lipschitz function which depends on several independent Gaussian variables and it is stated and proved using probabilistic arguments. The particularization of this result to  functions defined on the sphere $\ese^n(1)$ is equivalent to the concentration of measure of the values of any continuous function defined on this sphere $\ese^n(1)$ around its {\em Levy mean}, (see Lemma 2.3 in \cite{MS}), which in its turn is a corollary deduced from Theorem \ref{measurecon}.

Then, using the cosinus of the distance function to a fixed point in the sphere as our Lipschitz function and applying Theorem \ref{thB} we obtain:

\begin{theorem}\label{teo:mitjana}
 Given the $\epsilon$-strip $\Omega_\epsilon$ around the equator $E$, ($\epsilon \in (0,\pi/2)$),  there exists an absolute constant, (in the sense that do not depends on $n$), $\delta>0$ such that
\begin{equation}\label{eq3}
\frac{{\rm vol}\left(\Omega_\epsilon\right)}{{\rm vol}\left(\mathbb{S}^n(1)\right)}\geq 1-4e^{-\delta\bar\epsilon^2(n+1)}
\end{equation}
 where $\bar\epsilon=\sin(\epsilon) \in (0,1)$ and hence,
 \begin{equation}\label{eq3.1}
\lim_{n\to \infty}\frac{{\rm vol}\left(\Omega_\epsilon\right)}{{\rm vol}\left(\mathbb{S}^n(1)\right)}=1.
\end{equation}   
\end{theorem}

 As  the statement of our main result is directly inspired in Theorem \ref{teo:mitjana} and for the sake of completeness we will prove this theorem in the Appendix \ref{appendix}.

\begin{remark}
 We would like to draw attention to the following observation, which only apparently constitutes a paradox that would refute the last property: let us consider  small $\epsilon >0$. Then, by virtue of the limit \eqref{eq3.1}, as the dimension increases, the equatorial band occupies more and more area of the sphere, so that the two complementary big caps to the band have a small volume compared to the total volume of the sphere. 
 
 Let us now consider a rotation of the equatorial strip, (an isometry of the sphere acting on the strip), in such a way that the image due to this rotation of the strip is a domain in the sphere that cuts the two previous big caps. Since the image by the isometry of the band is a domain with the same area as the band and it occupies a large amount of area of the sphere, all this area will be concentrated outside the intersection with the caps, (which have small area). 
  The conclusion is that the entire area of the band will be concentrated in two small squares, antipodal and disconnected, which are the result of intersecting the old band prior to the rotation with the new one, (see picture below).
 \begin{figure}
     \centering
     \includegraphics[scale=0.3]{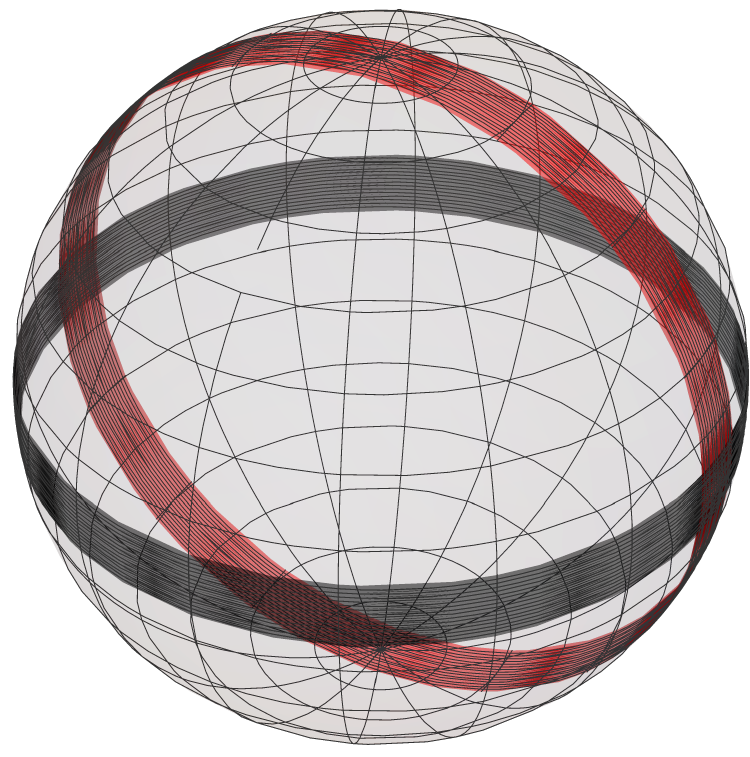}
     \caption{Two orthogonal equatorial bands which intersects in two small antipodal squares}
     \label{fig:enter-label}
 \end{figure}
 \begin{center}
\end{center}
 This fact not only does not contradict the corollary, but we must observe that, if we take, within one of these small squares, a spherical cap centered on it and circumscribed in it, we have that, when the dimension tends to infinity, the volume of this cap tends to zero while the area of the square tends to infinity.
\end{remark}

As we have mentioned before, we use the equatorial concentration of measure or \lq\lq fat equator effect'', as it has been described in Theorem \ref{teo:mitjana} to introduce the notion of  {\em extrinsic equatorial concentration of measure}  of a submanifold in the sphere. Let us consider $x: \Sigma^m \to \mathbb{S}^n(1)$ be a complete isometric immersion of the  $m$-dimensional manifold $\Sigma^m$ in the sphere $\mathbb{S}^n(1)$. Any totally geodesic equator $\mathbb{S}^{n-1}(1)$ on a sphere  $\mathbb{S}^n(1)$ can be described in terms of the intrinsic distance of its points to a given fixed point $p$, (and, in this sense, the equator $E$ that has served above to exemplify the phenomenon of concentration of the equatorial measure, is the set of points that are distant from the north pole by a distance equal to $\frac{\pi}{2}$). This description is given in the following way: given a point $p \in \mathbb{S}^n(1)$, \emph{the equator of $\mathbb{S}^n(1)$ with respect to $p$} is defined as 
   
 $$
    %\begin{aligned}
E(p):=\left\{x=(x_1,\cdots,x_{n+1})\in \mathbb{S}^n(1)\, :\,\,{\rm dist}^{\mathbb{S}^n(1)}(p,x)=\frac{\pi}{2} \right\},
%\end{aligned}
    $$
    where 
    $${\rm dist}^{\mathbb{S}^n(1)}(p,x)=\angle{(p,x)}=\arccos(p\cdot x)$$
    denotes the intrinsic distance in $\mathbb{S}^n(1)$ between the points $p$ and $x$. Then,   \emph{the strip around the equator $E(p)$} is defined as the set:

$$
\begin{aligned}
\Omega(p,\epsilon)&=\left\lbrace x\in \mathbb{S}^{n}(1)\, : \, \frac{\pi}{2}-\epsilon<{\rm dist}^{\mathbb{S}^{n}(1)}(p,x)=\angle{(p,x)}<\frac{\pi}{2}+\epsilon\right\rbrace\\&=\left\lbrace x\in \mathbb{S}^{n}(1)\, : \,   \cos\left({\rm dist}^{\mathbb{S}^{n}(1)}(p,x)\right)\in ( -\sin\epsilon,\sin\epsilon) \right\rbrace.
\end{aligned}
$$
\noindent Note that the strip around the equator $E$, $\Omega_\epsilon$, used in the estimations (\ref{eq2.5}) and (\ref{eq3}), is the strip $\Omega(p_N, \epsilon)$ where $p_N$ is the north pole of the sphere $\ese^n(1)$.

Our main result can be summarized asserting that there is an {\em extrinsic} equatorial concentration of measure when we consider minimal and compact submanifolds $\Sigma^m$ of any co-dimension in the Sphere $\mathbb{S}^n(1)$ and it is stated as follows:

\begin{theorem}\label{mainInt}
Let  $x: \Sigma^m \to \mathbb{S}^n(1)$ be an isometric and minimal immersion of the closed $m$-dimensional manifold $\Sigma$ in the sphere  $\mathbb{S}^n(1)$. Then, for any point $p \in \mathbb{S}^n(1)$ , any  $\epsilon \in (0,\pi/2)$ and any  absolute $\delta\in [0,\frac{1}{2})$, (in the sense that do not depends on $n$), we have 

\begin{equation}\label{eq4}
1\geq \frac{{\rm vol}\left(x^{-1}(\Omega(p,\epsilon))\right)}{{\rm vol}(\Sigma)}\geq  1- \sqrt{\frac{1}{1-2\delta}}e^{-\delta\bar\epsilon^2(m+1)}
\end{equation}
where $\bar\epsilon=\sin\epsilon \in (0,1)$. Hence, for each  $\epsilon \in (0,\pi/2)$,
\begin{equation}\label{eq5}
\lim_{m\to\infty}\frac{{\rm vol}\left(x^{-1}(\Omega(p,\epsilon))\right)}{{\rm vol}(\Sigma)}=1.
\end{equation}

\end{theorem}

We are going to finish this Introduction presenting a concentration of measure result of the fiber bundle projecting on the strip $\Omega(p, \epsilon)$ of a compact without boundary Riemannian submersion $\pi: \Sigma \rightarrow \ese^n(1)$ when this submersion is minimal, namely, when the (compact) fibers $\pi^{-1}(q) \subseteq \Sigma$ with $q \in \ese^n(1)$ are minimally immersed in $\Sigma$. This result, which can be considered a dual property of the extrinsic equatorial concentration of measure given by Theorem \ref{mainInt}, comes from both the structural duality of the equations that govern the geometry of the submersion with respect to the corresponding equations of the immersions and from the fact that the projection $\pi$ commutes with the Laplacians $\Delta^\Sigma$ and $\Delta^{\ese^n(1)}$ when the fibers of $\pi: \Sigma \rightarrow \ese^n(1)$ are minimal submanifolds of $\Sigma$.

\begin{theorem}\label{TeoSub}
Let  $\pi: \Sigma^m \to \mathbb{S}^n(1)$ be a Riemannian submersion with minimal fibers  of the closed $m$-dimensional manifold $\Sigma$ on to the sphere  $\mathbb{S}^n(1)$. Then, for any point $p \in \mathbb{S}^n(1)$ , any  $\epsilon \in (0,\pi/2)$ and any absolute $\delta\in [0,\frac{1}{2})$ we have 

\begin{equation}
1\geq \frac{{\rm vol}\left(\pi^{-1}(\Omega(p,\epsilon))\right)}{{\rm vol}(\Sigma)}=\frac{{\rm vol}\left(\Omega(p,\epsilon)\right)}{{\rm vol}\left(\mathbb{S}^n(1)\right)}\geq  1-\sqrt{\frac{\pi}{2}}e^{-\epsilon^2\frac{n-1}{2}}.
\end{equation}
Hence, for each  $\epsilon \in (0,\pi/2)$,
\begin{equation}
\lim_{n\to\infty}\frac{{\rm vol}\left(\pi^{-1}(\Omega(p,\epsilon))\right)}{{\rm vol}(\Sigma)}=1.
\end{equation}
\end{theorem}
\bigskip

We shall finish this Introduction with some general observations concerning the statements of Theorems \ref{mainInt} and \ref{TeoSub}.
\bigskip

\begin{remark}\label{remarkone}\

\begin{enumerate}
\item In contrast to the intrinsic approach, in the extrinsic case given by Theorem \ref{mainInt}, the proof of inequality \eqref{eq4} and equality \eqref{eq5} is based on the divergence theorem, rather than on the isoperimetric inequality, (as it is the case of the proof of inequality (\ref{eq2.5}) and equality (\ref{eq2.5.1}) in Theorem \ref{equator}, proved as a corollary of Theorem \ref{measurecon}) or rather than on probabilistic arguments using the Gaussian distribution, (as it is done in Theorem \ref{teo:mitjana}, deduced as a direct application of Theorem \ref{thB}).
\medskip

\item We can recover the {\em intrinsic} equality  (\ref{eq3.1}) in Theorem \ref{teo:mitjana} as a corollary of the {\em extrinsic} equality (\ref{eq5}) in Theorem \ref{mainInt} , taking $\Sigma=\mathbb{S}^n(1)$. In fact, if $\Sigma=\mathbb{S}^n(1)$, then we can use $x$ as the identity map  and hence, 
$$x^{-1}(\Omega(p,\epsilon))=x(\Sigma) \cap \Omega(p,\epsilon) =\Omega(p,\epsilon)$$

\noindent so, from equality (\ref{eq5}), we obtain
$$
\lim_{n\to\infty}\frac{{\rm vol}\left(x^{-1}(\Omega(p,\epsilon))\right)}{{\rm vol}(\Sigma)}=\lim_{n\to\infty}\frac{{\rm vol}\left(\Omega(p,\epsilon)\right)}{{\rm vol}(\mathbb{S}^n(1))}=1
$$
%\noindent Hence, we have an alternative way to prove the intrinsic equatorial concentration of measure in the sphere which is not based in the isoperimetric inequality.
But it is important to remark that inequality (\ref{eq3}) in Theorem \ref{teo:mitjana} is not the same nor it is  a consequence of inequality (\ref{eq4}) in Theorem \ref{mainInt} nor, on the other hand, does inequality \eqref{eq2.5} can  be derived directly from \eqref{eq3}. Note that the constants $\delta$ are not the same in Theorem \ref{mainInt} and Theorem \ref{teo:mitjana}.
\medskip

\item When we consider $\Sigma=\mathbb{S}^n(1)$ in Theorem \ref{mainInt}, we can compare inequality \eqref{eq4}  with inequality \eqref{eq3} in Theorem \ref{teo:mitjana}. To do this comparison, we must have into account that the constants $\delta$ are not the same in both theorems, (which, on the other hand, limits the scope of this comparison), and, moreover, that while $\delta$ is a fixed constant in Theorem  \ref{teo:mitjana}, in Theorem \ref{mainInt} the values of the constant $\delta$ ranges over the interval $[0,\frac{1}{2})$. If we start our comparison considering the bound  \eqref{eq3} in Theorem \ref{teo:mitjana}, namely, the bound 
$$\frac{{\rm vol}\left(\Omega_\epsilon\right)}{{\rm vol}\left(\mathbb{S}^n(1)\right)}\geq 1-4e^{-\delta\bar\epsilon^2(n+1)}$$
\noindent we can to choose the accurate value of $\delta=\delta_0$ in Theorem  \ref{mainInt} in order to have the bound $1- \sqrt{\frac{1}{1-2\delta_0}}e^{-\delta_0\bar\epsilon^2(n+1)}=1- 4e^{-\delta_0\bar\epsilon^2(n+1)}$. This value is
$
\delta_0=\frac{15}{32}
$ and inequality \eqref{eq4} will be 
$$\frac{{\rm vol}\left(x^{-1}(\Omega(p,\epsilon))\right)}{{\rm vol}(\ese^n(1))}\geq  1- 4e^{-\frac{15}{32}\bar\epsilon^2(n+1)}.$$
\medskip

\item  It should rest open to obtain, using our techniques, the statement of an {\em intrinsic} concentration of measure for compact and minimal submanifolds of the sphere, namely, if $A \subseteq \Sigma$ is a domain with ${\rm vol}(A) \geq \frac{1}{2} {\rm vol}(\Sigma)$, then the $\epsilon$-fattening of $A$ in $\Sigma$ satisfies the inequality
\begin{equation*}
1 \geq \frac{{\rm vol}(A_\epsilon)}{{\rm vol}(\Sigma)}\geq 1-\sqrt{\frac{\pi}{8}}e^{-\epsilon^2(m-1)/2}
\end{equation*}
\noindent and hence, for all $\epsilon \in [0,\frac{\pi}{2}]$, 
\begin{equation*}
\lim_{m\to \infty} \frac{{\rm vol}(A_\epsilon)}{{\rm vol}(\Sigma)}= 1
\end{equation*}

Here, the $\epsilon$-fattening of $A$ in $\Sigma$ should be defined as
$$A_\epsilon:=\left\{x\in \Sigma\, :\, {\rm dist}^{\Sigma}(x,A)<\epsilon\, \right\}.
$$

In fact, there is a result in this vein given as a Corollary of Theorem 6.9 in \cite{MS}, where an intrinsic concentration of measure in the line of Theorem \ref{measurecon} is stated in terms of the {\em first Dirichlet eigenvalue} of a compact and connected Riemannian manifold. In this sense, and if we consider a compact and minimal submanifold in the sphere $x:\Sigma^m \rightarrow \ese^n(1)$ as a manifold in itself, then, applying Theorem 6.9 we obtain, for any domain $A\subset \Sigma$ with $\frac{{\rm vol}(A)}{{\rm vol}( \Sigma)}=a$, the following concentration of measure
\begin{equation}\label{eq:Lambda}
   \frac{{\rm vol}\left(A_\epsilon\right)}{{\rm vol}( \Sigma)}\geq 1-(1-a^2)e^{-\epsilon\sqrt{\lambda_1(\Sigma)}\ln(1+a)}, 
\end{equation}
\noindent where $\lambda_1(\Sigma)$ is the first (non-zero) eigenvalue for the Laplacian of $\Sigma$.  
\medskip

\item We should note at this point that the concentration of measure (\ref{eq:Lambda}) is not related, a priori, with the dimension of the submanifold, so we can't conclude the description of the behaviour at infinity of the volume of the fattenings or the strips around the equators given by equalities \eqref{eq2}, \eqref{eq2.5.1}, \eqref{eq3.1} and \eqref{eq4}. However, looking again to inequality (\ref{eq:Lambda}), the things should change if we could obtain estimations for $\lambda_1(\Sigma)$ in terms of the dimension of the submanifold. In this vein, in problem 100 of \cite{YauMR0645728}, S.T Yau conjectured that the first eigenvalue of a compact minimal embedded hypersurface $\Sigma\subset \mathbb{S}^{n}$ is given by the dimension of $\Sigma$, namely
$$\lambda_1(\Sigma)=n-1.$$
If  Yau's conjecture is true it would  imply, by using inequality \eqref{eq:Lambda} that for any domain $A\subset \Sigma\subset \mathbb{S}^{n}$  of a minimal $n-1$-dimensional hypersurface, with 
 $\frac{{\rm vol}(A)}{{\rm vol}( \Sigma)}\geq \frac{1}{2}$ there will be two constants $c_1,c_2\in \mathbb{R}_{>0}$ such that
 \begin{equation}\label{eq:yau1}
   \frac{{\rm vol}\left(A_\epsilon\right)}{{\rm vol}( \Sigma)}\geq 1-c_1e^{-c_2\epsilon\sqrt{n-1}}. 
\end{equation}
\noindent so we should have an intrinsic concentration of measure for $\Sigma$ which could be considered a Wirtinger's type inequality for compact and minimal embedded hypersurfaces of the sphere, (see \cite{Gromov}).

The rate of decline with respect to the dimension in this result  should be slower than our rate of decrease on the dimension given by \eqref{eq4} but it should be true for other domains than any extrinsic fattening of the equator.
\medskip

\item The Yau's conjecture remains unsolved but it is known to be true in several  cases: for instance for homogeneous hypersurfaces (see \cite{MR0742598}, \cite{MR0814078}) and for certain isoparametric minimal hypersurfaces (see \cite{MR3080491}).   On the other hand, Choi and Wang proved in \cite{Choi1983559} that for compact, embedded  and minimal hypersurfaces $\Sigma\subset\mathbb{S}^{n+1}$ there exists the following lower bound in the first eigenvalue
$$\lambda_1(\Sigma)\geq \frac{n}{2}.$$
This lower bound is enough to state the concentration of measure provided by inequality \eqref{eq:yau1} in this particular context, namely, when we consider compact, embedded  and minimal hypersurfaces $\Sigma\subset\mathbb{S}^{n+1}$. Observe  that this concentration phenomena would be attained for fattenings of domains of any kind in compact, embedded and minimal hypersurfaces of the sphere, in contrast to the results presented in this paper which would be valid for a specific type of equatorial domains in the otherwise broader context of minimal and compact immersions of any codimension in the sphere. Observe moreover that, as we have noted before, the  rate of decline in the concentration of measure provided by inequality \eqref{eq:yau1}  is slower than the rate of decrease provided  by  inequality \eqref{eq4}.
\item Note that, when the case of Riemannian submersions $\pi: \Sigma \rightarrow \ese^n(1)$ with minimal fibers is considered, and unlike what happens with minimal and embedded immersions of codimension 1 in the sphere, (see observation (6) in Remark \ref{remarkone}), the first eigenvalue of the Laplacian is not bounded from below in general. 
Let us take for instance the family of Riemannian manifolds $\Sigma_\delta:=(\mathbb{S}^n(1)\times \mathbb{S}^1, g_{\mathbb{S}^{n}(1)}+\delta^2 g_{\mathbb{S}^{1}(1)})$, and the family of Riemannian submersions $\pi_\delta: \Sigma_\delta \rightarrow \ese^n(1)$ given by
$$
(p,q)\mapsto \pi_\delta(p,q)=p \,\,\forall \delta \in \erre_{>0}
$$
All these Riemannian submersions are minimal, indeed totally geodesic ones, and, as $\Sigma_\delta:=(\mathbb{S}^n(1),g_{\mathbb{S}^{n}(1)}) \times (\mathbb{S}^1,\delta^2g_{\mathbb{S}^{1}(1)})$  for any $\delta \in \erre$ we have that
$$\lambda_1(\Sigma_\delta)\leq {\rm min}\{\lambda_1(\mathbb{S}^n(1),g_{\mathbb{S}^{n}(1)}),\lambda_1(\mathbb{S}^1,\delta^2g_{\mathbb{S}^{1}(1)})\}={\rm min}\{n, \frac{1}{\delta^2}\}$$
\noindent so, for any $a\in (0,n]$ there is a value of $\delta$ such that
$$
\lambda_1(\Sigma_\delta)<a.
$$
Thence, in contrast to what happens for minimal hypersurfaces embedded in the sphere, there are no lower bounds for the first eigenvalue of a compact manifold which admits a Riemannian submersion to the sphere with minimal fibers.
\end{enumerate}
\end{remark}

%%%%%%%%%%%%%%%%%
% \subsection{Outline}
%%%%%%%%%%%%%%%%%
\subsection{Outline}\

The structure of the paper can be outlined as follows: after this Introduction where the notions related with the intrinsic and extrinsic equatorial concentration in the sphere has been presented, as well as a first glimpse of our main results, (Theorems \ref{mainInt} and \ref{TeoSub}), we shall follow in the following Section \S.2 with the study of the relative position of the closed and minimal submanifolds in the sphere with respect the equators of this sphere, proving a two-piece property which is satisfied by these closed and minimal submanifolds. We have proved, using the same techniques, the corresponding result which is satisfied by the total space $\Sigma$ of a Riemannian submersion. In Section \S.3, we shall state and prove our main results, and we have added an Appendix I with a proof of Theorem \ref{teo:mitjana}, because although we think that it is a folk result, we have not found a proof of it in the literature of the field.

%%%%%%%%%%%%%%%%%
% \subsection{Acknowledgement}
%%%%%%%%%%%%%%%%%
%\subsection{Acknowledgement}\

%The authors  wish to thank...

%%%%%%%%%%%%%%%%%%
%Section: Preliminaires
%%%%%%%%%%%%%%%%%%%%
\section{A high-dimensional two-piece property for parabolic and minimal submanifolds of the sphere $\mathbb{S}^n(1)$ }\

%%%%%%%%%%%%%%%%%%%%%%%%%%%%%%
%\subsection{Equator of a Sphere with respect a point}\
%%%%%%%%%%%%%%%%%%%%%%%%%%%%%%
%\medskip

We start with the definition of the {\em equator with respect a point} in the sphere $\ese^n(1)$, and the {\em strip around an equator }.

\begin{definition}
Given a point $p \in \mathbb{S}^n(1)$, \emph{the equator of  $\mathbb{S}^n(1)$ with respect to $p$} is defined as 
   
  $$
    \begin{aligned}
E(p):&=\left\{x=(x_1,\cdots,x_{n+1})\in \mathbb{S}^n(1)\, :\,\,{\rm dist}^{\mathbb{S}^n(1)}(p,x)=\frac{\pi}{2} \right\}\\&=
\left\{(x_1,\cdots,x_{n+1})\in \mathbb{S}^n(1)\, : \,\,\, p_1x_1+\cdots+p_{n+1}x_{n+1}=0\right\}
\end{aligned}
    $$
    \end{definition}
    
    In its turn, the {\em strip around an equator } is defined as:
    
\begin{definition}
The strip around the equator $E(p)$ is defined as the set:
    $$
\begin{aligned}
\Omega(p,\epsilon)&=\left\lbrace x\in \mathbb{S}^{n}(1)\, : \, \frac{\pi}{2}-\epsilon<{\rm dist}^{\mathbb{S}^{n}(1)}(p,x)=\angle{(p,x)}<\frac{\pi}{2}+\epsilon\right\rbrace\\&=\left\lbrace x\in \mathbb{S}^{n}(1)\, : \,   \cos\left({\rm dist}^{\mathbb{S}^{n}(1)}(p,x)\right)\in ( -\sin\epsilon,\sin\epsilon) \right\rbrace
\end{aligned}
$$
Hence, as $\epsilon \in (0,\frac{\pi}{2})$, if we put $\bar \epsilon=\sin \epsilon \in (0,1)$, then
$$
\Omega(p,\epsilon) =\left\lbrace x\in \mathbb{S}^{n}(1)\, : \, -\bar \epsilon<\cos\left({\rm dist}^{\mathbb{S}^{n}(1)}(p,x)\right)<\bar \epsilon\right\rbrace.
$$
\end{definition}

With these definitions at hand, it is evident that for any point $q\in E(p)$, 
        $$
{\rm dist}^{\mathbb{S}^n(1)}(p,q)=\frac{\pi}{2}={\rm dist}^{\mathbb{S}^n(1)}(-p,q)
        $$
where $-p$ denotes the antipodal point of $p \in \mathbb{S}^{n}(1)$. It is also evident  that any equator in the sphere $\mathbb{S}^n(1)$ is an equator with respect to some $p$.
        
     Let us consider now $x: \Sigma^m \to \mathbb{S}^n(1)$  a complete isometric immersion of the  $m$-dimensional manifold $\Sigma^m$ in the sphere.  A natural question that arises in this context is to ask ourselves about the position of the submanifold with respect to the equators of the ambient sphere. 
     
     In fact, when $\Sigma=\mathbb{S}^{1}(1)$ is a great circle isometrically immersed in the sphere $\mathbb{S}^{2}(1)$, it is not hard to show that then $\Sigma$ intersects {\em all} the equators of $ \mathbb{S}^{2}(1)$, (which are the great circles of the $2$-dimensional sphere).
     
     When the dimension of the ambient sphere is $3$, A. Ros showed in \cite{R} a {\em Two-piece property}, which implies a more general intersection result to that described above for the two-dimensional case. Namely, every equator of the 3-sphere $\mathbb{S}^3(1)$ divides each embedded closed (compact, without boundary) non-totally geodesic minimal surface in exactly two open connected pieces.

  Therefore,  a natural question which arises in this context is following: under what conditions, (topological, analytical, curvature restrictions, etc), it is guaranteed, (for any dimension of the ambient sphere and any co-dimension of the immersed submanifold), that the submanifold cuts all the equators of the ambient sphere, \emph{i.e.},
        $$x(\Sigma) \cap E(p) \neq \emptyset\,\,\,\forall p \in \mathbb{S}^n(1)$$ and hence, will intersect all the $\epsilon$-strips
        $$x(\Sigma) \cap \Omega(p,\epsilon) \neq \emptyset, (\text{i.e.}, x^{-1}(\Omega(p,\epsilon))\neq \emptyset\,\,\,\forall p \in \mathbb{S}^n(1))$$
   
    Our first result  takes the form of a kind of a \lq\lq high dimensional two-piece property", and describes the position of the submanifold with respect the equators of $\mathbb{S}^n(1)$ and their strips, when this submanifold is minimal and parabolic:
 
\begin{theorem}\label{twopiecegeneral}

Let $x:\Sigma^m \to \mathbb{S}^n(1)$ be a complete,  minimal and isometric immersion, and let $p$ be any point of $\mathbb{S}^n(1)$. Suppose that $\Sigma$ is parabolic. Then
    \begin{itemize}
        \item Either $x(\Sigma)\subseteq E(p)$, 
        \item or, $x(\Sigma)\cap E(p)\neq \emptyset$
            \end{itemize}
            and  hence, in any case,  $x(\Sigma) \cap \Omega(p,\epsilon) \neq \emptyset$, (i.e. $x^{-1}(\Omega(p,\epsilon))\neq \emptyset$).
\end{theorem}

\begin{proof}
Let $K$ be a connected component of $\Sigma$. Since $\Sigma$ is parabolic, $K$ is a parabolic manifold as well.
Given $p\in \mathbb{S}^n(1)$, the \emph{extrinsic distance function} with respect to $p$ is the function given by
$$
\begin{aligned}
r_p:K \to \mathbb{R}, \quad q\mapsto r_p(q):&={\rm dist}^{\mathbb{S}^n(1)}(p,x(q))\\&=\angle{(p,x(q))}=\arccos(p\cdot x(q))
\end{aligned}
$$
so, given $q \in K$, with $x(q)=(x_1,...,x_{n+1}) \in \mathbb{S}^n(1)$, we have that
$$
\cos(r_p(q))=p\cdot x(q) = p_1x_1+\cdots+p_{n+1}x_{n+1}
$$
 and hence, $\cos\circ r_p$ is smooth on $K$.   We need the following
\begin{lemma}\label{coslema}
Let $x:K^m \to \mathbb{S}^n(1)$ be a complete and minimal isometric immersion, and let $p$ be any point of $\mathbb{S}^n(1)$. Then, for all $q \in K$,
 \begin{equation}\label{cos}
%\begin{aligned}
\Delta^K \cos(r_p(q))=-m\cos(r_p(q))
%\end{aligned}
\end{equation}
\end{lemma}
\begin{proof}
  Applying Takahasi's Lemma, (see \cite{T}), we have,  as $x$ is minimal, that 
  $$\Delta^K x_i=-m x_i\,\,\forall i=1,...,n+1$$
  so we conclude 
$$
\begin{aligned}
\Delta^K \cos(r_p(q))=&p_1\Delta^\Sigma x_1+\cdots+p_{n+1}\Delta^\Sigma x_{n+1}\\=&-m\left(p_1x_1+\cdots+p_{n+1}x_{n+1}\right)=-m\cos(r_p(q))
\end{aligned}
$$\end{proof}
To prove  the Theorem, first, let us observe that, if the extrinsic distance function is constant,
$\cos\circ r_p(q)=c \,\,\forall q \in \Sigma$, then 
$$0=\Delta^K \cos(r_p(q))=-m\cos\circ r_p(q)\,\forall q \in K,$$
so $r_p(q)=\frac{\pi}{2} \,\,\forall q \in \Sigma$ and hence, in this case we have that $$x(K) \subseteq E(p).$$

This implies that the extrinsic distance function is a constant function in $K$ if and only if $x(K)\subseteq E(p)$.
Now,  let us suppose that $x(K) \not\subseteq E(p)$ and that $x(K)\cap  E(p)=\emptyset$. Then $\cos\circ r_p$ is not constant on $K$ and, moreover, as $K$ is connected,  if $x(K)\cap  E(p)=\emptyset$ then $r_p(q) \in [0,\frac{\pi}{2})\,\,\forall q \in K$ or $r_p(q) \in (\frac{\pi}{2},\pi]\,\,\forall q \in K$. To see this last assertion, let us note that, in case there are points $q_1, q_2 \in K$ such that $r_p(q_1) \in [0,\frac{\pi}{2})$ and  $r_p(q_2) \in (\frac{\pi}{2},\pi]$, then we can conclude that there are points $q_1, q_2 \in K$ such that $\cos(r_p(q_1)) >0$ and  $\cos (r_p(q_2)) <0$, and therefore, $K$ is not connected because $K=\{ q \in K: \cos(r_p(q)) >0\}\cup \{ q \in K: \cos(r_p(q))< 0\} $. We can conclude therefore that $\cos\circ r_p$ is a bounded non-constant subharmonic, (superharmonic), function defined on $K$, which is a contradiction with the parabolicity of $K$. Hence, if  $x(K) \not\subseteq E(p)$, then $x(K)\cap  E(p)\neq \emptyset$, which implies that $x(\Sigma)\cap  E(p)\neq \emptyset$.\end{proof}
Note that, given the compact, (and hence, parabolic), and minimal submanifold $\Sigma$, we can apply  the above theorem \ref{twopiecegeneral} to conclude the following corollary:
\begin{corollary}\label{cortwo}
	Let  $x: \Sigma^m \to \mathbb{S}^n(1)$ be an isometric and minimal immersion of a closed, (compact without boundary), $m$-dimensional manifold $\Sigma$ in the sphere  $\mathbb{S}^n(1)$, and let $p$ be any point of $\mathbb{S}^n(1)$.  Then
    \begin{itemize}
        \item Either $x(\Sigma)\subseteq E(p)$, 
        \item or, $x(\Sigma)\cap E(p)\neq \emptyset$
            \end{itemize}
            and  hence, in any case,  $x(\Sigma) \cap \Omega(p,\epsilon) \neq \emptyset$, (i.e. $x^{-1}(\Omega(p,\epsilon))\neq \emptyset$).
\end{corollary}
%In conclusion, in this first result, we have seen that when $x: \Sigma^m \to \mathbb{S}^n(1)$ is  compact and minimal, then \emph{always} there are points of $x(\Sigma)$ included in any equatorial strip, some in one part of the equator, some in the other part.

We can show an analogous result for Riemannian submersions on the sphere, using the fact, alluded in the introduction, that the projection $\pi$ commutes with the Laplacians $\Delta^\Sigma$ and $\Delta^{\ese^n(1)}$ when the fibers of $\pi: \Sigma \rightarrow \ese^n(1)$ are minimal submanifolds of $\Sigma$.

In this case, the set   $\pi^{-1}(\Omega(p,\epsilon))=\cup_{q \in \Omega(p,\epsilon)} \pi^{-1}(q) $ is the fiber bundle projecting on the equatorial strip $ \Omega(p,\epsilon)$, and we want to know if this set is not empty for every $p \in \ese^n(1)$ and $\epsilon >0$. To show this fact, we will prove that the \lq\lq shadow" of the total manifold $\Sigma$ projected by $\pi$ on the sphere cuts all the equators $E(p)$ in the sphere, namely, that
 $$\pi(\Sigma) \cap E(p) \neq \emptyset\,\,\,\forall p \in \mathbb{S}^n(1)$$ and hence, will intersect all the $\epsilon$-strips
        $$\pi(\Sigma) \cap \Omega(p,\epsilon) \neq \emptyset$$
        \noindent so we eventually have  
        $$\pi^{-1}(\Omega(p,\epsilon))\neq \emptyset\,\,\,\forall p \in \mathbb{S}^n(1)$$
   
These questions are answered in the following

\begin{theorem}\label{twopiecegeneralsub}

Let $\pi:\Sigma^m \to \mathbb{S}^n(1)$ be a complete, minimal and parabolic Riemannian submersion, and let $p$ be any point of $\mathbb{S}^n(1)$.  Then
    \begin{itemize}
        \item Either $\pi(\Sigma)\subseteq E(p)$, 
        \item or, $\pi(\Sigma)\cap E(p)\neq \emptyset$
            \end{itemize}
            and  hence, in any case,  $\pi(\Sigma) \cap \Omega(p,\epsilon) \neq \emptyset$, (i.e. $\pi^{-1}(\Omega(p,\epsilon))\neq \emptyset$).
\end{theorem}
The proof of Theorem \ref{twopiecegeneralsub} follows the lines of the proof of Theorem \ref{twopiecegeneral},  when we consider an isometric immersion but now using the function $\cos(r_p)$ with $r_p:=d_p \circ \pi$ and $d_p$ the geodesic distance in $\mathbb{S}^n(1)$ to $p\in \mathbb{S}^n$ . We take into account in this context that since the fibers of $\pi$ are minimal submanifolds, then, (see \cite{W}, \cite{GP})
$$\Delta^\Sigma(f\circ \pi)(q)=\bigg(\big(\Delta^{\ese^n(1)}f\big)\circ\pi\bigg)(q)$$
  for any smooth function $f:\mathbb{S}^n\to \mathbb{R}$. Choosing $f$ as the geodesic distance $d_p$ to $p\in \mathbb{S}^n(1)$ we can state, see Theorem 27 of \cite{Petersen} for instance, 
  $$
  \Delta^\Sigma r_p=\Delta^\Sigma (d_p\circ \pi)=\big(\Delta^{\ese^n(1)}d_p\big)\circ \pi=(n-1)\frac{\cos(d_p)}{\sin(d_p)}\circ \pi=(n-1)\frac{\cos(r_p)}{\sin(r_p)}
  $$
  Hence for any smooth function $F:\mathbb {R}\to \mathbb{R}$,
  $$
\Delta^\Sigma F(r_p)={\rm div }(F'(r_p)\nabla^\Sigma r_p)=F''(r_p)\Vert \nabla^\Sigma r_p\Vert^2+ (n-1)\frac{\cos(r_p)}{\sin(r_p)} F'(r_p).
$$
  
By expressing $\nabla^\Sigma r$ in terms an orthonormal basis of the horizontal distribution and then projecting on the base space using the fact that $r_p=d_p\circ \pi$ we conclude that $\Vert \nabla^\Sigma r_p\Vert^2=1$ and hence 
\begin{equation}\label{eq:lapsub}
    \Delta^\Sigma F(r_p)=F''(r_p)+(n-1) \frac{\cos(r_p)}{\sin(r_p)} F'(r_p).
\end{equation}
In particular
$$
\Delta^\Sigma \cos(r_p)=-n \cos(r_p),
$$
what is everything needed for the proof of the above theorem. 

As in the case of immersions, where we obtain Corollary \ref{cortwo} for the compact case as a direct consequence of Theorem \ref{twopiecegeneral}, we have the following corresponding corollary for compact submersions deduced as a direct application of Theorem \ref{twopiecegeneralsub}:
\begin{corollary}\label{cortwoSub}
	Let $\pi:\Sigma^m \to \mathbb{S}^n(1)$ be a complete, minimal and compact Riemannian submersion, and let $p$ be any point of $\mathbb{S}^n(1)$.  Then
    \begin{itemize}
        \item Either $\pi(\Sigma)\subseteq E(p)$, 
        \item or, $\pi(\Sigma)\cap E(p)\neq \emptyset$
            \end{itemize}
            and  hence, in any case,  $\pi(\Sigma) \cap \Omega(p,\epsilon) \neq \emptyset$, (i.e. $\pi^{-1}(\Omega(p,\epsilon))\neq \emptyset$).
\end{corollary}
%%%%%%%%%%%%%%%%%%
\section{Main results}
%%%%%%%%%%%%%%%%%%%%

In the above Section, we have proved in Corollaries \ref{cortwo} and \ref{cortwoSub} that, when the immersion $x: \Sigma^m \to \mathbb{S}^n(1)$ is  compact and minimal, (and respectively, the submersion $\pi:\Sigma^m \to \mathbb{S}^n(1)$ is compact and minimal), then \emph{always} there are points of $x(\Sigma)$, (or $\pi(\Sigma)$), included in any equatorial strip, some in one part of the equator, some in the other part. The question now is to estimate the \lq\lq amount" of such points (in $x(\Sigma)$ or $\pi(\Sigma)$),  which are included in these equatorial strips. 

We shall see that, as occurs in the intrinsic case, as the dimension increases, the \lq\lq greater'' the number of points of the sets $x(\Sigma)$ or $\pi(\Sigma)$ within the equatorial strips. Namely, we observe a \emph{measure concentration phenomenon} satisfied by $x: \Sigma\to \mathbb{S}^n(1)$ and $\pi:\Sigma^m \to \mathbb{S}^n(1)$.

The following first two  theorems, (Theorem \ref{extrinsicconc1} and Theorem \ref{extconc2}), describes a \lq\lq Fat equator effect" in a closed and minimal submanifold in the sphere, the latter being a refinement of the first with respect to the rate of measure concentration. Indeed, while in the first theorem this rate is polynomial, in the second theorem it will be shown that the concentration of measure grows at an exponential rate, as in Theorem \ref{measurecon} and its corollary  Theorem \ref{equator}. Although the second theorem improves clearly the first one, assuming the same hypotheses, we show both in order to make visible the techniques used in both results.

We shall finish this Section stating and proving in Theorem \ref{TeoSub1}  that, when the total space $\Sigma$ in a minimal Riemannian submersion $\pi: \Sigma \rightarrow \ese^n(1)$ is compact, almost all measure in $\Sigma$ concentrates in the the fiber bundle of the fibers projecting on the equatorial strips $\Omega(p, \epsilon)$ of any width $\epsilon$ in the sphere $\ese^n(1)$. As we have mentioned in the Introduction, its proof is based in the same techniques than we have used for the immersions.

\begin{theorem}[Main I]\label{extrinsicconc1}
Let  $x: \Sigma^m \to \mathbb{S}^n(1)$ be an isometric and minimal immersion of the closed $m$-dimensional manifold $\Sigma$ in the sphere  $\mathbb{S}^n(1)$. Then, for any point $p \in \mathbb{S}^n(1)$ and any  $\epsilon \in (0,\pi/2)$ we have 

$$
 1 \geq \frac{{\rm vol}\left(x^{-1}(\Omega(p,\epsilon))\right)}{{\rm vol}(\Sigma)}\geq1-\frac{1}{(m+1)\sin^2\epsilon}.
$$ Hence, for each $p$ and $\epsilon$,
$$
\lim_{m\to\infty}\frac{{\rm vol}\left(x^{-1}(\Omega(p,\epsilon))\right)}{{\rm vol}(\Sigma)}=1.
$$
\end{theorem}
\begin{proof}\

Given the extrisic distance function $r_p: \Sigma \rightarrow \erre$, we know that

\begin{equation}\label{eqOne}
\begin{aligned}
\Div^\Sigma\big(\cos (r_p(q))\,\nabla^\Sigma \cos(r_p(q))\big)=&\cos(r_p(q))\big(\Delta^\Sigma\cos(r_p(q))\big)\\&+\Vert \nabla^\Sigma \cos(r_p(q))\Vert^2.
\end{aligned}
\end{equation}
Then, since $\Delta^\Sigma \cos(r_p(q))=-m\cos(r_p(q))$, and integrating equation \eqref{eqOne} along $\Sigma$, we have, using that $\Sigma$ is compact, that $\cos\circ r_p$ is smooth in $\Sigma$ and the divergence theorem:

\begin{equation}\label{eq:longa}
    \begin{aligned}
m\int_\Sigma\cos^2(r_p)=&\int_\Sigma\Vert \nabla^{\Sigma}  \cos(r_p)\Vert^2-\int_\Sigma \Div^\Sigma\big(\cos(r_p)\nabla^\Sigma \cos(r_p)\big)\\
=&\int_\Sigma \sin^2(r_p)\Vert \nabla^{\Sigma} r_p\Vert^2\\\leq& \int_\Sigma \sin^2(r_p)
={\rm vol}(\Sigma)- \int_\Sigma \cos^2(r_p).
\end{aligned}
\end{equation}
Then
\begin{equation}\label{eq:nula}
(m+1)\int_\Sigma \cos^2(r_p)\leq {\rm vol}(\Sigma).    
\end{equation}
Which implies, as  $\Sigma-x^{-1}(\Omega(p,\epsilon)) \subseteq \Sigma$,
\begin{equation}\label{eq:unua}
(m+1)\int_{\Sigma-x^{-1}(\Omega(p,\epsilon))} \cos^2(r_p)\leq {\rm vol}(\Sigma)    
\end{equation}
The strip around the equator $E(p)$ is ($\bar \epsilon=\sin \epsilon$):
$$
\Omega(p,\epsilon) =\left\lbrace x\in \mathbb{S}^{n}(1)\, : \, -\bar \epsilon<\cos\left({\rm dist}^{\mathbb{S}^{n}(1)}(p,x)\right)<\bar \epsilon\right\rbrace.
$$
Since $r_p(q):={\rm dist}^{\mathbb{S}^{n}(1)}(p,x(q))$, the set $x^{-1}(\Omega(p,\epsilon))$ can be written as 
$$
x^{-1}(\Omega(p,\epsilon))=\left\lbrace q \in \Sigma\,:\, -\bar \epsilon<\cos\left(r_p(q))\right)<\bar \epsilon\right\rbrace.
$$ Then for all $q \in\Sigma-x^{-1}(\Omega(p,\epsilon))$, we have
$ \bar\epsilon^2 \leq \cos^2(r_p(q))$  and by using inequality \eqref{eq:unua} we conclude that
 $$
(m+1)\bar \epsilon^2\left({\rm vol}(\Sigma)-{\rm vol}(x^{-1}(\Omega(p,\epsilon)))\right)\leq {\rm vol}(\Sigma),
$$ which finally implies the statement of the Theorem:
$$
\frac{{\rm vol}(x^{-1}(\Omega(p,\epsilon)))}{{\rm vol}(\Sigma)}\geq 1-\frac{1}{(m+1)\bar\epsilon^2}.
$$\end{proof}

\begin{theorem}[Main II]\label{extconc2}
Let  $x: \Sigma^m \to \mathbb{S}^n(1)$ be an isometric and minimal immersion of the closed $m$-dimensional manifold $\Sigma$ in the sphere  $\mathbb{S}^n(1)$. Then, for any point $p \in \mathbb{S}^n(1)$ , any  $\epsilon \in (0,\pi/2)$ and any $\delta\in [0,\frac{1}{2})$ we have 

\begin{equation}\label{eq4Bis}
1\geq \frac{{\rm vol}\left(x^{-1}(\Omega(p,\epsilon))\right)}{{\rm vol}(\Sigma)}\geq  1- \sqrt{\frac{1}{1-2\delta}}e^{-\delta\bar\epsilon^2(m+1)}
\end{equation}
where $\bar\epsilon=\sin\epsilon$.  Hence, for each $p \in \ese^n(1)$ and $\epsilon>0$,
$$
\lim_{m\to\infty}\frac{{\rm vol}\left(x^{-1}(\Omega(p,\epsilon))\right)}{{\rm vol}(\Sigma)}=1
$$
\end{theorem}
\begin{proof}\

We need the following

\begin{lemma}\label{keylemaimm}
Let  $x: \Sigma^m \to \mathbb{S}^n(1)$ be an isometric and minimal immersion of the closed $m$-dimensional manifold $\Sigma$ in the sphere  $\mathbb{S}^n(1)$ and let $p\in \mathbb{S}^n(1)$. Then

\begin{equation}\label{lemmaformula1}
\int_\Sigma \cos^{2k}(r_p) \leq \frac{\int_{\mathbb{S}^m(1)} \cos^{2k}(\widetilde{r}_p)}{{\rm vol}(\mathbb{S}^m(1))} {\rm vol}(\Sigma)
\end{equation}
where $\widetilde{x}: \mathbb{S}^m(1) \rightarrow \mathbb{S}^n(1)$ is a totally geodesic immersion of $\ese^m(1)$ in $\ese^n(1)$, $r_p$ is the extrinsic distance function to $p$ given by the immersion $x$, and $\widetilde{r}_p$ is the extrinsic distance to $p$ on $\ese^n(1)$ given by the totally geodesic immersion $\widetilde{x}$. As a consequence, we have the following inequality,  for a given $t>0$:

\begin{equation}\label{lemmaformula2}
%\begin{aligned}
\int_{\Sigma} e^{t \cos^2r_p}  \leq\frac{\int_{\ese^m(1)} e^{t \cos^2\widetilde{r}_p}}{{\rm vol}(\mathbb{S}^m(1))} {\rm vol}(\Sigma)
%\end{aligned}
\end{equation}
\end{lemma}
\begin{proof}

Let us apply an inductive argument to prove inequality (\ref{lemmaformula1}): we are going to see first that this inequality is true for $k=1$.  From \eqref{eq:longa} using that $\Vert \nabla^\Sigma r_p\Vert\leq 1$ we conclude inequality \eqref{eq:nula}:
$$
\int_\Sigma \cos^2(r_p) \leq \frac{1}{m+1} {\rm vol}(\Sigma)
$$
Now, if we consider the totally geodesic submanifold $\widetilde{x}: \mathbb{S}^m(1) \rightarrow \mathbb{S}^n(1)$, we conclude the following equality, (in an analogous way  to \eqref{eq:longa}, but using now that $\Vert\nabla^{\mathbb{S}^m(1)} \widetilde{r}_p\Vert=1$, because $\widetilde{x}$ is a totally geodesic immersion), 
\begin{equation}\label{eqcinc}
\int_{\mathbb{S}^m(1)} \cos^2(\widetilde{r}_p) = \frac{1}{m+1} {\rm vol}(\mathbb{S}^m(1)).
\end{equation}

Hence,  we have 
\begin{equation}\label{eqsis}
\int_\Sigma \cos^2(r_p) \leq \frac{\int_{\mathbb{S}^m(1)} \cos^2(\widetilde{r}_p)}{{\rm vol}(\mathbb{S}^m(1))} {\rm vol}(\Sigma)
\end{equation}
and we have proved that the inequality (\ref{lemmaformula1}) is true for $k=1$.

Assuming that it is true for $k >1$, namely, that
\begin{equation}\label{eqsisBis}
\int_\Sigma \cos^{2k}(r_p) \leq \frac{\int_{\mathbb{S}^m(1)} \cos^{2k}(\widetilde{r}_p)}{{\rm vol}(\mathbb{S}^m(1))} {\rm vol}(\Sigma),
\end{equation}
we see that it holds for $k+1$ in the following way: by using lemma \ref{coslema}  we have  for any $q\in \Sigma$ that

\begin{equation}
\begin{aligned}
&\Div^\Sigma\left(\cos^{2k+1} (r_p(q))\,\nabla^\Sigma \cos(r_p(q)\right)=\cos^{2k+1}(r_p(q))\Delta^\Sigma\cos(r_p(q)\\&+(2k+1) \cos^{2k}(r_p(q))\Vert \nabla^\Sigma \cos(r_p(q))\Vert^2\\
&=-m\cos^{2k+2}(r_p(q))\\&+(2k+1) \cos^{2k}(r_p(q))\Vert \nabla^\Sigma \cos(r_p(q))\Vert^2.
\end{aligned}
\end{equation}
Then, integrating the above equation  along $\Sigma$, we have, using that $\Sigma$ is compact and the divergence theorem:

\begin{equation}\label{eqdossBis}
\begin{aligned}
m\int_\Sigma \cos^{2(k+1)}(r_p)=&(2k+1)\int_\Sigma\cos^{2k}(r_p)\Vert \nabla^{\Sigma}  \cos(r_p)\Vert^2\\&-\int_\Sigma \Div^\Sigma\left(\cos^{2k+1}(r_p))\nabla^\Sigma \cos(r_p)\right)
\\
=&(2k+1)\int_\Sigma \cos^{2k}(r_p)\sin^2(r_p)\Vert \nabla^{\Sigma} r_p\Vert^2\\
=&(2k+1)\int_\Sigma \cos^{2k}(r_p)\left(1-\cos^2(r_p)\right)\Vert \nabla^{\Sigma} r_p\Vert^2\\
\leq& (2k+1)\int_\Sigma \cos^{2k}(r_p)-(2k+1)\int_\Sigma \cos^{2k+2}(r_p).
\end{aligned}
\end{equation}
Thence
\begin{equation}\label{eqdoss2Bis}
\begin{aligned}
\int_\Sigma \cos^{2(k+1)}(r_p)\leq &\frac{2k+1}{m+2k+1}\int_\Sigma \cos^{2k}(r_p)\\ 
\leq&  \frac{2k+1}{m+2k+1}\frac{\int_{\mathbb{S}^m(1)} \cos^{2k}(\widetilde{r}_p)}{{\rm vol}(\mathbb{S}^m(1))} {\rm vol}(\Sigma)
\end{aligned}
\end{equation}

On the other hand, same computations than in (\ref{eqdossBis}), but when we consider the totally geodesic immersion of $\ese^m(1)$ in $\ese^n(1)$ gives the formula

$$
\begin{aligned}
m\int_{\ese^m(1)} \cos^{2(k+1)}(\widetilde{r}_p)=&(2k+1)\int_{\ese^m(1)}\cos^{2k}(\widetilde{r}_p)-(2k+1)\int_{\ese^m(1)} \cos^{2k+2}(\widetilde{r}_p)
\end{aligned}
$$
and hence
\begin{equation}\label{eqdoss4Bis}
%\begin{aligned}
\int_{\ese^m(1)} \cos^{2(k+1)}(\widetilde{r}_p)=\frac{2k+1}{m+2k+1}\int_{\ese^m(1)}\cos^{2k}(\widetilde{r}_p)
%\end{aligned}
\end{equation}
Finally, using (\ref{eqdoss2Bis}) and (\ref{eqdoss4Bis}) we obtain

\begin{equation}\label{eqdoss5Bis}
\begin{aligned}
\int_\Sigma \cos^{2(k+1)}(r_p)\leq& \frac{2k+1}{m+2k+1}\int_\Sigma \cos^{2k}(r_p)\\  \leq & \frac{2k+1}{m+2k+1}\frac{\int_{\mathbb{S}^m(1)} \cos^{2k}(\widetilde{r}_p)}{{\rm vol}(\mathbb{S}^m(1))} {\rm vol}(\Sigma)\\
=&\frac{\int_{\mathbb{S}^m(1)} \cos^{2(k+1)}(\widetilde{r}_p)}{{\rm vol}(\mathbb{S}^m(1))} {\rm vol}(\Sigma)
\end{aligned}
\end{equation}

Now, we are going to prove inequality (\ref{lemmaformula2}): given a fixed $t>0$, we have, applying inequality (\ref{lemmaformula1}) that

$$\int_\Sigma\frac{t^k \cos^{2k}(r_p)}{k!} \leq \frac{\int_{\mathbb{S}^m(1)}\frac{t^k \cos^{2k}(\widetilde{r}_p)}{k!}}{{\rm vol}(\mathbb{S}^m(1))} {\rm vol}(\Sigma)$$
 so, applying dominated convergence theorem and the power series expansion of the exponential function, we have, for a given $t>0$:
$$
\begin{aligned}
\int_{\Sigma} e^{t \cos^2r_p} =&\sum_{k=0}^\infty \int_\Sigma\frac{t^k \cos^{2k}(r_p)}{k!} \\
\leq& 
\sum_{k=0}^\infty \frac{\int_{\mathbb{S}^m(1)}\frac{t^k \cos^{2k}(\widetilde{r}_p)}{k!}}{{\rm vol}(\mathbb{S}^m(1))} {\rm vol}(\Sigma)=\frac{\int_{\ese^m(1)} e^{t \cos^2\widetilde{r}_p}}{{\rm vol}(\mathbb{S}^m(1))} {\rm vol}(\Sigma)
\end{aligned}
$$
\end{proof}
Now, as we have seen in the proof of Theorem \ref{extrinsicconc1}
we have that for all  $q \in\Sigma-x^{-1}(\Omega(p,\epsilon))$, 
$$ \bar\epsilon^2 \leq \cos^2(r_p(q))$$
and hence, for every $t >0$, and for all $q \in \Sigma- x^{-1}(\Omega(p,\epsilon))$,
$$e^{t\bar\epsilon^2} \leq e^{t\cos^2(r_p(q))} $$

Therefore, using inequality (\ref{eqsis}), we conclude
 \begin{equation}\label{eqsisBis}
 \begin{aligned}
e^{t\bar\epsilon^2}&\int_{\Sigma-x^{-1}(\Omega(p,\epsilon))} 1\leq 
\int_{\Sigma-x^{-1}(\Omega(p,\epsilon))}e^{t \cos^2r_p}\\&\leq \int_\Sigma e^{t\cos^2r_p} \leq \frac{\int_{\mathbb{S}^m(1)}e^{t \cos^2\widetilde{r}_p}}{{\rm vol}(\mathbb{S}^m(1))} {\rm vol}(\Sigma)
\end{aligned}
\end{equation}
Then we can state that
\begin{equation}\label{preskauxlasta}
    1-\frac{{\rm vol}(x^{-1}(\Omega(p,\epsilon))}{{\rm vol}(\Sigma)}\leq e^{-t\bar\epsilon^2}\frac{\int_{\mathbb{S}^m(1)}e^{t \cos^2\widetilde{r}_p}}{{\rm vol}(\mathbb{S}^m(1))}.
\end{equation}
Now we need the following 
\begin{lemma}\label{lastalemma} For any $0<t<\frac{m+1}{2}$,
$$
  \frac{\int_{\mathbb{S}^m(1)}e^{t \cos^2\widetilde{r}_p}}{{\rm vol}(\mathbb{S}^m(1))}\leq \sqrt{\frac{m+1}{m+1-2t}}.
    $$
\end{lemma}
\begin{proof}
    Let us define the function 
    $$
s\mapsto F(s):= \frac{\int_{\mathbb{S}^m(1)}e^{s \cos^2\widetilde{r}_p}}{{\rm vol}(\mathbb{S}^m(1))}.
    $$
  By the dominated convergence theorem we know that $F(0)=1$ and
     
    $$
F'(s)=\frac{\int_{\mathbb{S}^m(1)}\cos^2\widetilde{r}_pe^{s \cos^2\widetilde{r}_p}}{{\rm vol}(\mathbb{S}^m(1))}
    $$
 By using lemma \ref{coslema}, the divergence theorem, and that since the immersion is totally geodesic $\Vert \nabla\widetilde{r}_p\Vert^2=1$ we have 
 $$
\begin{aligned}
mF'(s)=&\frac{-\int_{\mathbb{S}^m(1)}\cos\widetilde{r}_pe^{s \cos^2\widetilde{r}_p}\Delta \cos\widetilde{r}_p}{{\rm vol}(\mathbb{S}^m(1))}
=\frac{-\int_{\mathbb{S}^m(1)}{\rm div}\left(\cos\widetilde{r}_pe^{s \cos^2\widetilde{r}_p}\nabla \cos\widetilde{r}_p\right)}{{\rm vol}(\mathbb{S}^m(1))}\\&+\frac{\int_{\mathbb{S}^m(1)}\left\langle\nabla \left(\cos\widetilde{r}_pe^{s \cos^2\widetilde{r}_p}\right),\nabla \cos\widetilde{r}_p\right\rangle}{{\rm vol}(\mathbb{S}^m(1))}\\
=&\frac{\int_{\mathbb{S}^m(1)}\sin^2\widetilde{r}_pe^{s \cos^2\widetilde{r}_p}}{{\rm vol}(\mathbb{S}^m(1))}+2s\frac{\int_{\mathbb{S}^m(1)}\sin^2\widetilde{r}_p{ \cos^2\widetilde{r}_p}e^{s \cos^2\widetilde{r}_p}}{{\rm vol}(\mathbb{S}^m(1))}\\
=&\frac{\int_{\mathbb{S}^m(1)}e^{s \cos^2\widetilde{r}_p}}{{\rm vol}(\mathbb{S}^m(1))}-\frac{\int_{\mathbb{S}^m(1)}\cos^2\widetilde{r}_pe^{s \cos^2\widetilde{r}_p}}{{\rm vol}(\mathbb{S}^m(1))}+2s\frac{\int_{\mathbb{S}^m(1)}\sin^2\widetilde{r}_p{ \cos^2\widetilde{r}_p}e^{s \cos^2\widetilde{r}_p}}{{\rm vol}(\mathbb{S}^m(1))}\\
\leq & F(s)-F'(s)+2sF'(s).
    \end{aligned}
 $$
 Thence, for $s<\frac{m+1}{2}$,
 $$
\frac{F'(s)}{F(s)}\leq \frac{1}{m+1-2s}
 $$
 Integrating the above inequality between $0$ and $t$ we conclude that
 $$
\ln F(t)\leq \ln \sqrt{\frac{m+1}{m+1-2t}}.
 $$
 and the lemma is proved.
\end{proof}
Taking now $t=\delta(m+1)$ with $0\leq \delta<\frac{1}{2}$, the above lemma and inequality \eqref{preskauxlasta} we conclude that
\begin{equation}\label{lasta}
     1-\frac{{\rm vol}(x^{-1}(\Omega(p,\epsilon))}{{\rm vol}(\Sigma)}\leq e^{-\delta\bar\epsilon^2(m+1)}\sqrt{\frac{1}{1-2\delta}}.
\end{equation}
and the theorem is proved.\end{proof}

Finally, we prove that, in a Riemannian submersion of the sphere, the measure concentrates in the fiber bundles which projects inside its equatorial strips:

\begin{theorem}\label{TeoSub1}
Let  $\pi: \Sigma^m \to \mathbb{S}^n(1)$ be a Riemannian submersion with minimal fibers  of the closed $m$-dimensional manifold $\Sigma$ on to the sphere  $\mathbb{S}^n(1)$. Then, for any point $p \in \mathbb{S}^n(1)$ , any  $\epsilon \in (0,\pi/2)$ and any $\delta\in [0,\frac{1}{2})$ we have 

\begin{equation}
1\geq \frac{{\rm vol}\left(\pi^{-1}(\Omega(p,\epsilon))\right)}{{\rm vol}(\Sigma)}=\frac{{\rm vol}\left(\Omega(p,\epsilon)\right)}{{\rm vol}\left(\mathbb{S}^n(1)\right)}\geq  1-\sqrt{\frac{\pi}{2}}e^{-\epsilon^2\frac{n-1}{2}}.
\end{equation}
Hence, for each  $\epsilon$,
\begin{equation}
\lim_{n\to\infty}\frac{{\rm vol}\left(\pi^{-1}(\Omega(p,\epsilon))\right)}{{\rm vol}(\Sigma)}=1.
\end{equation}
\end{theorem}
\begin{proof}
Since $\pi:\Sigma \rightarrow \ese^n(1)$ is a Riemmanian submersion we have that
$$
\Vert \nabla^\Sigma r_p(x)\Vert =\Vert \nabla\Sigma \left(d_p\circ \pi\right)\Vert=1, \quad x\in \Sigma-\{\pi^{-1}(p)\}\cup \{\pi^{-1}(-p) \}
$$
where  $d_p$ is the distance function on $\mathbb{S}^n(1)$ to $p$. 
Let us denote by $D_R$ the sublevel sets of $r_p$, namely
$$
D_R:=\{x\in \Sigma\, : \, r_p(x)<R\}=\pi^{-1}(B_R(p))
$$
where $B_R(p)$ is the geodesic ball of radius $R$ of $\mathbb{S}^n(1)$ centered at $p$. By the coarea formula the function $t\mapsto U(t)={\rm vol}(D_t)$ has derivative
$$
U'(t)=\int_{\partial D_t}\frac{1}{\Vert \nabla^\Sigma r_p\Vert}={\rm vol}(\partial D_t)
$$
for any $t\in (0,\pi)$. Now, we consider the function $E:\Sigma\to \mathbb{R}$ given by
$$
q\in \Sigma\mapsto E(q)=\int_0^{r_p(q)}\frac{V(t)}{V'(t)}dt
$$
where $V(t)={\rm vol}(B_t(p))$, and $V'(t)=\frac{d}{dt}V(t)={\rm vol }(\partial B_t(p))$. As  $E=E(r_p)$ is a radial function, we have, using equation \eqref{eq:lapsub}, that
$$
\Delta^{\Sigma}E=1.
$$
Hence by using the divergence theorem we can deduce that
$$
\begin{aligned}
    U(t)=\int_{D_t}\Delta^\Sigma E=\frac{V(t)}{V'(t)}\int_{\partial D_t}\Vert \nabla r_p\Vert=\frac{V(t)}{V'(t)}U'(t).
\end{aligned}
$$
Thence,
$$
\frac{d}{dt}\left(\frac{U(t)}{V(t)}\right)=0
$$
for all $t\in (0,\pi)$. This implies that for any $R\in(0,\pi)$
$$
\frac{{\rm vol}(\pi^{-1}(B_R(p)))}{{\rm vol}(B_R(p))}=\lim_{R\to \pi}\frac{{\rm vol}(\pi^{-1}(B_R(p)))}{{\rm vol}(B_R(p))}=\frac{{\rm vol}(\Sigma)}{{\rm vol}(\mathbb{S}^n(1))}.
$$
Finally the theorem follows by using
$$
\begin{aligned}
\frac{{\rm vol}(\pi^{-1}(\Omega(p,\epsilon)))}{{\rm vol}(\Sigma)}=&\frac{{\rm vol}(\pi^{-1}(B_{\frac{\pi}{2}+\epsilon}(p)))}{{\rm vol}(\Sigma)}-\frac{{\rm vol}(\pi^{-1}(B_{\frac{\pi}{2}-\epsilon}(p)))}{{\rm vol}(\Sigma)}  \\
=&\frac{{\rm vol}(B_{\frac{\pi}{2}+\epsilon}(p)))\frac{{\rm vol}(\pi^{-1}(B_{\frac{\pi}{2}+\epsilon}(p)))}{
{\rm vol}(B_{\frac{\pi}{2}+\epsilon}(p)))}}{{\rm vol}(\Sigma)}-\frac{{\rm vol}(B_{\frac{\pi}{2}-\epsilon}(p)))\frac{{\rm vol}(\pi^{-1}(B_{\frac{\pi}{2}-\epsilon}(p)))}{
{\rm vol}(B_{\frac{\pi}{2}-\epsilon}(p)))}}{{\rm vol}(\Sigma)}\\
=&\frac{{\rm vol}(B_{\frac{\pi}{2}+\epsilon}(p))}{{\rm vol}(\mathbb{S}^n(1))}-\frac{{\rm vol}(B_{\frac{\pi}{2}-\epsilon}(p))}{{\rm vol}(\mathbb{S}^n(1))}\\
=&\frac{{\rm vol}(\Omega(p,\epsilon))}{{\rm vol}(\mathbb{S}^n(1))},
\end{aligned}
$$
and the Theorem follows by using Theorem \ref{equator}.

\end{proof}

%%%%%%%%%%%%%%%%%%%%%%%%%%%%%%%%%%
%%%%%%%%%%%%%%%%%%%%%%%%%%%%%%%%%%%%
\section*{Appendix I: Proof of theorem \ref{teo:mitjana}}\label{appendix}
%%%%%%%%%%%%%%%%%%%%%%%%%%%%%%%%%%%%
%%%%%%%%%%%%%%%%%%%%%%%%%%%%%%%%%%%%%%%

Our theorem \ref{teo:mitjana}  is a consequence of Corollary V.2 in \cite{MS} where it is proved a concentration of measure around the averaged/mean value of a Lipschitz function defined on the sphere $\ese^n(1)$. To apply this result in our present context, we shall define, given a point $p\in \mathbb{S}^n$, the function $f_p: \ese^n(1) \rightarrow \erre^+$ as
$$
q\mapsto f_p(q):=\cos(d_p(q))=\langle p, q\rangle 
$$
\noindent where $d_p(q):={\rm dist}^{\ese^n(1)}(p,q)={\rm arccos}(\langle p, q\rangle)$ is the distance in the sphere from $p$ to $q$.
The function $f_p$ is a Lipschitz function with Lipschitz constant $1$ because given two points  $q,s\in \mathbb{S}^n$, we have that 

$$
\begin{aligned}
    \left\Vert f_p(q)-f_p(s)\right\Vert&=\left\Vert \cos(d_p(q))-\cos(d_p(s))\right\Vert=\Vert \langle p, q\rangle-\langle p, s\rangle \Vert \\&=\Vert \langle p, q- s\rangle \Vert
\leq \Vert q-s\Vert={\rm dist}^{\mathbb{R}^{n+1}}(q,s)
\leq  {\rm dist}^{\mathbb{S}^n}(p,q).
\end{aligned}
$$
Moreover since $\Delta_{\mathbb{S}^{n}}\cos(d_p)=-n\cos(d_p)$ by using the divergence theorem we conclude that the mean of $f_p$ in $\ese^n(1)$ is 
$$
\langle f_p \rangle=\langle \cos(d_p)\rangle:=\frac{\int_{\mathbb{S}^n}\cos(d_p)}{{\rm vol}(\mathbb{S}^n)}=0.
$$

Now, let us consider the strip, (putting $\bar \epsilon:= \sin \epsilon$),
$$\Omega(p,\epsilon)=\left\lbrace q\in \mathbb{S}^{n}(1)\, : \,   \cos\left(d_p(q)\right)\in ( -\sin\epsilon,\sin\epsilon) \right\rbrace=\{q\in \mathbb{S}^n\, : \, \vert \cos(d_p(q))\vert < \bar\epsilon \}$$
\noindent  so
$$\begin{aligned}
{\rm vol}(\Omega(p, \epsilon))&={\rm vol}(\{q\in \mathbb{S}^n\, : \, \vert \cos(d_p(q))\vert < \bar\epsilon \})\\&={\rm vol}(\ese^n(1))-{\rm vol}(\{q\in \mathbb{S}^n\, : \, \vert \cos(d_p(q))\vert \geq \bar\epsilon \})
\end{aligned}$$
Now, applying corollary V.2 of \cite{MS}  to the Lipschitz function $f_p=\cos\circ d_p$ with Lispchitz constant $\sigma=1$, we have that 
$$\frac{{\rm vol}(\{q\in \mathbb{S}^n\, : \, \vert \cos(d_p(q))\vert \geq \bar\epsilon \})}{{\rm vol}(\mathbb{S}^n)}\leq 4 e^{-\delta\bar\epsilon^2(n+1)}$$
\noindent  and hence
%$$
%\frac{{\rm vol}(\Omega(p,\epsilon))}{{\rm vol}(\mathbb{S}^n)}=\frac{{\rm vol}\left(\{q\in \mathbb{S}^n\, : \, \vert \cos(d_p(q))\vert \geq \bar\epsilon \}\right)}{{\rm vol}(\mathbb{S}^n)}\geq 1-4e^{-\delta \bar\epsilon^2(n+1)}
%$$
$$
\begin{aligned}
\frac{{\rm vol}(\Omega(p,\epsilon))}{{\rm vol}(\mathbb{S}^n)}&=\frac{ {\rm vol}(\ese^n(1))-{\rm vol}(\{q\in \mathbb{S}^n\, : \, \vert \cos(d_p(q))\vert \geq \bar\epsilon \})}{{\rm vol}(\mathbb{S}^n) }\\&=1-\frac{{\rm vol}(\{q\in \mathbb{S}^n\, : \, \vert \cos(d_p(q))\vert \geq \bar\epsilon \}) }{ {\rm vol}(\ese^n(1))}\geq 1-4e^{-\delta \bar\epsilon^2(n+1)}
\end{aligned}
$$
and the theorem is proved.
%%%%%%%%%%%%%%%%%%%%%%%%%%%%%%%%%%%%%
%%%%%%%%%%%%%%%%%%%%%%%%%%%%%%%%%%%%%

  %%%%%%%%%%%%%%%%

\end{document}